\documentclass[11pt, a4paper]{article}
\usepackage{amssymb,amsmath,amsthm,tocloft}
\usepackage{fancyhdr}
\usepackage[british]{babel}
\usepackage{enumitem}

\usepackage[usenames,dvipsnames,table]{xcolor}
\usepackage{tikz}
\usetikzlibrary{decorations.pathreplacing}
\usepackage{bbm}
\newcounter{propcounter}

\usepackage{hyperref}

\usepackage{todonotes} 

\usepackage[margin=2.5cm]{geometry}

\newtheorem{theorem}{Theorem}[section]
\newtheorem{prop}[theorem]{Proposition}
\newtheorem{lemma}[theorem]{Lemma}

\newtheorem{conj}[theorem]{Conjecture}

\newcommand{\repeatlabel}{}
\newtheorem*{repeatlemma}{Lemma \repeatlabel}

\theoremstyle{definition}

\newtheorem{defn}[theorem]{Definition}
\newtheorem{claim}[theorem]{Claim}

\newcommand{\ep}{\varepsilon}

\newcommand{\cF}{\mathcal{F}}

\newcommand{\De}{\Delta}

\title{Clique immersion in graphs without a fixed bipartite graph}
\author{Hong Liu\thanks{Extremal Combinatorics and Probability Group (ECOPRO), Institute for Basic Science (IBS), Daejeon, South
Korea, Email: {\tt hongliu@ibs.re.kr}, supported by the Institute for Basic Science (IBS-R029-C4).}, ~~~~~~~Guanghui Wang\thanks{School of Mathematics, Shandong University, Jinan, China, Email: {\tt ghwang@sdu.edu.cn}, supported by Natural Science Foundation of China (11631014) and Shandong University Multidisciplinary Research and Innovation Team of Young Scholars.}, ~~~~~~~Donglei Yang\thanks{\textbf{Correspondence}:~Data Science Institute, Shandong University, Shandong, China, Email: {\tt dlyang@sdu.edu.cn}, supported by the China Postdoctoral Science Foundation (2021T140413), Natural Science Foundation of China (12101365) and Natural Science Foundation of Shandong Province (ZR2021QA029).}
}

\date{}
\begin{document}

\maketitle
\begin{abstract}
	A graph $G$ contains $H$ as an \emph{immersion} if there is an injective mapping $\phi: V(H)\rightarrow V(G)$ such that for each edge $uv\in E(H)$, there is a path $P_{uv}$ in $G$ joining vertices $\phi(u)$ and $\phi(v)$, and all the paths $P_{uv}$, $uv\in E(H)$, are pairwise edge-disjoint. An analogue of Hadwiger's conjecture for the clique immersions by Lescure and Meyniel, and independently by Abu-Khzam and Langston, states that every graph $G$ contains $K_{\chi(G)}$ as an immersion.
We prove that for any constant $\ep>0$ and integers $s,t\ge2$, there exists $d_0=d_0(\ep,s,t)$ such that every $K_{s,t}$-free graph $G$ with $d(G)\ge d_0$ contains a clique immersion of order $(1-\ep)d(G)$. This implies that the above-mentioned  conjecture is asymptotically true for graphs without a fixed complete bipartite graph.
\end{abstract}



\section{Introduction}
A graph $H$ is a \emph{minor} of another graph $G$ if $H$ can be obtained from  $G$ via vertex deletions, edge deletions and edge contractions. A conjecture of Hadwiger \cite{hdwg} states that every graph $G$ with $\chi(G)\ge t$ contains $K_{t}$ as a minor. This conjecture is widely open for $\chi(G)\ge7$; the case $\chi(G)=5$ is equivalent to the celebrated Four-Color Theorem and the case $\chi(G)=6$ was solved by Robertson, Seymour and Thomas \cite{robert}.
A graph $H$ is a \emph{topological
minor} of another graph $G$ if there is an injective mapping $\phi: V(H)\rightarrow V(G)$ such that for each edge $uv\in E(H)$, there is a path $P_{uv}$ in $G$ joining vertices $\phi(u)$ and $\phi(v)$, and all the paths $P_{uv}$, $uv\in E(H)$, are pairwise internally vertex-disjoint. A stronger conjecture proposed by Haj\'{o}s in 1940's \cite{hajos} states that every graph $G$ contains $K_{\chi(G)}$ as a topological minor. However, this conjecture
is known to be false in general: Catlin \cite{cat} disproved this conjecture for all $\chi(G)\ge 7$. It is also natural to consider graphs with given average degree. In this direction, Kostochka \cite{kos} and independently, Thomason \cite{thom} proved that average
degree $d(G)=\Omega(t\sqrt{\log t})$ in a graph $G$ forces $K_t$ as a minor, and this bound is optimal. This remained the best order of magnitude for Hadwiger's conjecture until very recent breakthroughs by Norin, Postle and Song \cite{nps} and Postle \cite{pos}.

In this paper, we consider immersions,  first introduced by Nash-Williams \cite{nash}. Given two graphs $G$ and $H$, we say $G$ contains an $H$-\emph{immersion} if there exists an injective  mapping $\phi: V(H)\rightarrow V(G)$ such that for each edge $uv\in E(H)$, there is a path $P_{uv}$ in $G$ connecting $\phi(u)$ and $\phi(v)$; and
all the paths $P_{uv}$, $uv\in E(H)$, are pairwise edge-disjoint. We call
the vertices $\{\phi(v)\mid v\in V(H)\}$ the \emph{branch} vertices of the immersion.  As a weakening of topological minor, the immersion relation requires paths to be pairwise edge-disjoint rather than vertex-disjoint. Although graph minors and graph immersions are incomparable, Robertson and Seymour \cite{robert2} showed that graphs are well-quasi-ordered by immersion, analogous to their
celebrated graph minors project. An immersion variant of Hadwiger's conjecture was proposed by Lescure and Meyniel \cite{les-mey} \footnote{They in fact conjectured a stronger statement that the $K_t$-immersion can be embedded so that every path $P_{uv}$, $uv\in E(K_t)$ is internally vertex disjoint from the set of branch vertices.} in $1989$, and independently, by Abu-Khzam and Langston \cite{abu} in $2003$.
\begin{conj}\emph{\cite{abu, les-mey}}\label{conj-hadwiger-immersion}
	Every graph $G$ with $\chi(G)\ge t$ contains $K_t$ as an immersion.
\end{conj}

Conjecture \ref{conj-hadwiger-immersion} has
seen more success than the Hadwiger's Conjecture: the cases $t\le 4$ are trivial, while the cases $5\le t\le7$ are proved by DeVos, Kawarabayashi, Mohar and Okamura \cite{devos}. Here for any $t\in \mathbb{N}$, we use $f(t)$ to denote the least integer such that every graph $G$ with $\delta(G)\ge f(t)$ contains $K_t$ as an immersion. In \cite{devos}, they showed that $f(t)=t-1$ holds for any $t\in\{5,6,7\}$,  It is easy to see that $f(t)\ge t-1$.  For $t\ge 8$, there are infinitely many constructions \cite{collin,devos1} showing that $f(t)\ge t$. A linear upper bound on $f(t)$ is due to DeVos, Dvo\v{r}\'{a}k, Fox, McDonald, Mohar and Scheide \cite{devos1}, showing that $f(t)\le 200t$, and it was then improved to $11t+7$ by Dvo\v{r}\'{a}k and Yepremyan \cite{dvo}, who asked whether for all $t\ge8$, $f(t)=t$. Building on the work of \cite{dvo}, Gauthier, Le and
Wollan \cite{gau} showed that $f(t)\le7t+7$.

Our work is motivated by a result of K\"{u}hn and Osthus \cite{kuhn}. They proved that for any fixed bipartite graph $H$, Hadwiger's conjecture holds strongly for any $H$-free graph $G$ by constructing a clique minor of order polynomially larger than $d(G)$. Note that here and throughout the paper, by $H$-free we mean that there is no subgraph in $G$ that is isomorphic to $H$. Improved bounds on the order of clique minor in $H$-free graphs for more general bipartite graphs $H$ were later obtained by
Krivelevich and Sudakov \cite{kri} and Norin, Postle and Song \cite{nps}.

Our main result reads as follows. It in particular implies that Conjecture \ref{conj-hadwiger-immersion} is asymptotically true if we forbid any fixed complete bipartite graph.

\begin{theorem}\label{main}
Given any positive constant $\ep$ and integers $s,t\ge2$, there exists $d_0=d_0(\ep,s,t)$ such that every $K_{s,t}$-free graph $G$ with $d(G)\ge d_0$ contains a clique immersion of order $(1-\ep)d(G)$.
\end{theorem}

The bound above is asymptotically optimal as $G$ could be $d$-regular. It would be interesting to improve on the additive error term.

\medskip

Our approach differs from previous works on immersions in~\cite{devos1, dvo, gau} which generally reduce the problem on embedding clique immersions to a dense eulerian graph via suppressing vertices and use some list coloring arguments. Our proof adopts a more direct embedding approach, making use of certain expander and builds on the techniques developed in the work of Liu and Montgomery \cite{lm} on embedding clique subdivisions. In \cite{lm}, to embed large clique subdivisions in dense expander, a key idea is to build many unit structures by finding vertices with large boundaries. Then the arguments often reduce to greedily connecting the units with vertex-disjoint paths. We are attempting this approach to immersions which boils down to embedding edge-disjoint paths. Unlike in \cite{lm} where the problem is harder for dense expanders, the bulk of the work in our paper is to handle sparse expanders. In sparse expanders, one can still adapt the approach in \cite{lm} to obtain a clique immersion of order linear in $d$. As the expansion is sublinear, if a small constant portion of neighbors of a vertex is used, then the vertex can still expand well. To embed $K_{(1-o(1))d}$-immersion in sparse expanders, we essentially need vertices that can expand past a relatively large set of vertices even after deleting ninety-nine percent of its incident edges. However the sublinear expansion is not strong enough to guarantee the expansion of any such vertex. To overcome this issue, we instead use an idea from the work of Haslegrave, Kim and Liu \cite{hkl} to find many vertex-disjoint subexpanders and grow a vertex in each subexpander robustly until it reaches large enough size to enjoy further expansion in the main expander (see Section~\ref{subs22} for a more elaborate sketch of the proof).

\medskip

The rest of the paper will be organized as follows. In Section 2, we introduce some necessary notions and tools whilst outline the proof of our main result. We divide the main proof into two cases depending on the density of our expander.
Sections~\ref{sec-dense} and \ref{sec-sparse} are devoted to embedding large clique immersions in dense and sparse expanders, respectively.

\section{Preliminaries and notation}
For a set of vertices $X\subseteq V(G)$, denote its \emph{external neighbourhood} by $N_G(X):=\{u\in V(G)\setminus X: uv\in E(G) \mbox{ for some } v\in X\}$. Furthermore, denote by $\partial_G(X)$ the \emph{edge boundary} of $X$, i.e.~$E_G[X,V(G)\setminus X]$. Define $G-X$ to be the induced subgraph of $G$ on $V(G)\setminus X$ and for a subgraph $F$, use $G\setminus F$ to denote the spanning subgraph with $E(F)$ removed. Throughout the paper, the \emph{length} of a path always denotes the number of edges in the path. For two sets $A, B\subseteq V(G)$, an $(A,B)$-path is path $P$ with two endpoints separately lying in $A$ and $B$ such that $P$ does not have any interior vertices in $A$ or $B$. Moreover, the \emph{distance} between $A$ and $B$ is the minimum length of an $(A,B)$-path and if $A\cap B\neq \emptyset$, then the distance is zero.
For each $r\in\mathbb{N}$, the $r$-\emph{th sphere} around $X$, denoted by $N^{r}_G(X)$, is the set of vertices with distance exactly $r$ from $X$. So $N_G^0(X)=X$ and $N_G^1(X)=N_G(X)$. Denote by $B_G^r(X)$ the \emph{ball} of radius $r$ around $X$, i.e.~$B_G^r(X)=\cup_{0\le i\le r}N_G^i(X)$. Throughout the proof, all logarithms are in the natural basis.

Let $G$ be a $K_{s,t}$-free graph on $n$ vertices for some $2\le s \le t$ and $d(G)=d$. Then a classical result of K\H{o}v\'ari, S\'os and Tur\'an \cite{kov} on the Tur\'{a}n number of complete bipartite graphs tells that $e(G)=O(n^{2-1/s})$ and thus we have
\begin{equation}\label{lower}
n/d=\Omega(d^{\frac{1}{s-1}}).
\end{equation}
\medskip
We will also make use of the following bipartite version.

\begin{lemma}\emph{\cite{kov}}\label{bipartite}
For integers $s,t$ with $2\le s\le t$, there exists a constant $c$ such that the following holds. Let $G=(V_1,V_2,E)$ be a $K_{s,t}$-free bipartite graph with $|V_1|=n_1,|V_2|=n_2$. Then
  $$e(G)\le cn_1^{1-1/s}n_2.$$
\end{lemma}


\subsection{Robust sublinear expander}
For $\ep_1>0$ and $k>0$, let $\rho(x)$ be
the function
\begin{eqnarray}\label{epsilon}
\rho(x)=\rho(x,\ep_1,k):=\left\{\begin{tabular}{ l c r }
$0$ & \mbox{ ~if } $x<k/5$ \\
$\ep_1/\log^2(15x/k)$ & $\mbox{ ~~if  }  ~x\ge k/5$, \\
\end{tabular}
\right.
\end{eqnarray}
\noindent where, when it is clear from context we will not write the dependency on $\ep_1$ and $k$ of $\rho(x)$. Note that $\rho(x)\cdot x$ is increasing for $x\ge k/2$.
In \cite{kom}, Koml\'{o}s and Szemer\'{e}di introduced a notion of $(\ep_1,k)$-\emph{expander} $G$ in which for any subset $X\subseteq V(G)$ with $k/2\le |X|\le |V(G)|/2$, we have $|N_G(X)|\ge\rho(|X|)\cdot |X|$. In this paper, we shall utilize the following robust version  recently developed by Haslegrave, Kim and Liu \cite{hkl}.
\begin{defn}\cite{hkl}
A graph $G$ is an \emph{$(\ep_1,k)$-robust-expander} if for all subsets $X\subseteq V(G)$ of size $k/2\le |X|\le |V(G)|/2$ and any subgraph $F\subseteq G$ with $e(F)\le d(G)\cdot\rho(|X|)|X|$, we have that
\begin{eqnarray}\label{expansion}
|N_{G\setminus F}(X)|\ge \rho(|X|)\cdot |X|.
\end{eqnarray}
\end{defn}

We will use the following version of expander lemma in \cite{hkl}, which states that every graph contains a robust expander with almost the same average degree.

\begin{lemma}\emph{\cite{hkl}}\label{lem-expander}
Let $C>30$, $0<\ep_1\le \frac{1}{10C}, 0<\ep_2<1/2, d>0, \eta=C\ep_1/\log3$ and $\rho(x)=\rho(x,\ep_1,\ep_2 d)$ be as in~\eqref{epsilon}. Then every graph $G$ with $d(G)=d$ has a subgraph $G'$ that is an $(\ep_1,\ep_2 d)$-robust-expander with $d(G')\geq (1-\eta)d$ and $\delta(G')\geq d(G')/2$.
\end{lemma}
\noindent

The following small diameter property is the key property of the expanders that we will repeatedly make use of. It roughly says that we can find a relatively short path between any two large sets, avoiding a small set of vertices or edges.
\begin{lemma}[{Robust small diameter, Lemma 2.3 in \cite{hkl}}]\label{lem-diameter}
  Let $0<\ep_1, \ep_2<1$ and $G$ be an $n$-vertex $(\ep_1,\ep_2d)$-robust-expander. Given two sets $X_1,X_2\subseteq V(G)$ of size $x\ge \ep_2d/2$, let $Y$ be a vertex set of size at most $\rho(x)x/4$ and $F$ be a subgraph with at most $d(G)\rho(x)x$ edges. Then there is an $(X_1,X_2)$-path of length at most $\frac{2}{\ep_1}\log^3(15n/\ep_2d)$ in $(G\setminus F)-Y$.
\end{lemma}


The following is our main lemma, which finds in a robust expander a clique immersion of asymptotically optimal size.
\begin{lemma}\label{lem-immersion-not-dense}
Let $0<\ep_1\le1/400, 0<\ep_2<1/2, \eta\ge\max\{\frac{40\ep_1}{\log 3},5\ep_2\}$ and $2\le s\le t$. Then there exists $d_0$ satisfying the following. Let $G$ be a $K_{s,t}$-free $(\ep_1,\ep_2d)$-robust-expander of order $n$ and
     $d(G)=d\ge d_0.$ 
     Then $G$ contains a clique immersion of order at least $(1-9\eta)d$.
\end{lemma}

Theorem \ref{main} immediately follows from Lemmas \ref{lem-expander} and~\ref{lem-immersion-not-dense}.\medskip

\begin{proof}[Proof of Theorem \ref{main}.]
Given any constant $\ep>0$, we choose $C=40$, $\ep_1= \ep\log3/500$, $\ep_2<\min\{\ep/50,1/2\}$ and $\eta=\ep/10$. Then $\eta \ge\eta':=40\ep_1/\log3$ and Lemma \ref{lem-expander} applied to $G$ with $\ep_1,\ep_2,C=40, \eta'$, gives a subgraph $G'$ that is an $(\ep_1,\ep_2d)$-robust-expander with $d(G')\geq (1-\eta')d(G)\ge(1-\eta)d(G)$. By applying Lemma \ref{lem-immersion-not-dense} to $G'$ with $\ep_1,\ep_2,\eta=\ep/10$, we obtain a clique immersion of order at least $(1-9\eta)d(G')\ge(1-\ep)d(G)$ in $G'$, which is also an immersion in $G$.
\end{proof}

\subsection{Outline of the proof}\label{subs22}

Here we sketch an overview of the proof of Lemma~\ref{lem-immersion-not-dense}. We divide the proof into two cases according to
the density $d$ of our expander. For the dense case $d\ge \log^{200s} n$ in Section~\ref{sec-dense}, we adapt the approach of Liu and Montgomery \cite{lm} on embedding clique subdivisions in dense expanders, in which we construct a specific unit structure (see Definition~\ref{unit}) for embedding a $K_{d-o(d)}$-immersion. The main work is then to $(1)$ find in a dense expander $d-o(d)$ mutually edge-disjoint units with distinct centers (see Lemma~\ref{lem-units}); $(2)$ and connect these units with pairwise edge-disjoint paths so as to obtain a desired clique immersion.

For the proof of the sparse case, we divide the proof into two cases in Section~\ref{subs43} depending on the number of vertices of sufficiently large degree. Before that, we first deal with a special case when the maximum degree is bounded, which is covered by Lemma~\ref{bounded max-degree}.  Then we denote $Z_1$ as the set of high degree vertices (to be defined later in Section~\ref{subs43}). The case $|Z_1|\ge d$ is easy to handle and the strategy here is to greedily connect those vertices in $Z_1$ (see Claim \ref{many-large} in Section~\ref{subs42}). Towards the case when $G$ has less than $d$ such vertices, we instead focus our attention to the subgraph $G':=G-Z_1$. Using the assumption of $K_{s,t}$-freeness, we obtain that $G'$ has almost the same average degree with $G$ (see Claim \ref{cl1}).

Note that the subgraph $G'$ might not have the expansion property for small sets of vertices as in $G$, whereas we essentially need vertices that locally expand well for embedding a $K_{d-o(d)}$-immersion. To handle this, we then show in Section \ref{manyexpander} that there exist in $G'$ many dense $(\ep_1,\ep_2d)$-robust-expanders $F_1,F_2,\cdots,F_d$ which are relatively small and pairwise far apart from each other (Claim~\ref{manyexp}). Here we iteratively make use of Lemma~\ref{lem-expander} and Lemma~\ref{bipartite} to guarantee the existence of a dense expander one by one.

Anchoring at these small dense subexpanders, we first grow a vertex $v_i$ in each subexpander $F_i$ robustly into a small ball $K_i$ which has size relatively larger than $|Z_1|$ even after deleting ninety-nine percent of incident edges of $v_i$ (See Claim~\ref{roubao} in Section~\ref{kernels}). 
Then by further expanding (subsets of) each $K_i$ into two large balls in $G'$, we follow the strategy in the proof of Lemma~\ref{bounded max-degree} to finish the embedding of a desired clique immersion.



\section{Embedding immersions in dense expanders}\label{sec-dense}
In this section, we prove Lemma~\ref{lem-immersion-not-dense} assuming in addition that $d\ge \log^{200s} n$. Throughout the rest of this paper, we write
\begin{eqnarray}\label{mnd}
\ell=(1-5\eta)d,\quad \ell'=(1-4\eta)d,\quad  m:=\frac{2}{\ep_1}\log^3\left(\frac{15n}{\ep_2d}\right).
\end{eqnarray}
Note that by~\eqref{lower}, when $d$ is sufficiently large, then $n/d$ and also $m$ are sufficiently large, and
\begin{eqnarray}\label{eq-n-large}
n/d\ge m^{200} \quad\mbox{ and }\quad d\ge m^{50s}.
\end{eqnarray}
Also, for sufficiently large $d$, since $\rho(x)$ is decreasing in the interval $[\ep_2d/2,n]$, we have that for every $\ep_2d/2\le x\le n$,
\begin{eqnarray}\label{eq-eps-not-small}
     \rho(x)\ge \rho(n)\ge\frac{1}{m}.
\end{eqnarray}
\begin{defn}\label{unit}
Given integers $h_1,h_2,h_3>0$, an \emph{$(h_1,h_2,h_3)$-unit}~$F$ is a graph consisting of a center $v$, $h_1$ vertex-disjoint stars $S(u_i)$ centered at $u_i$, each of size $h_2$, and edge-disjoint $(v,u_i)$-paths, $i=1,\ldots,h_1$, each of length at most $h_3$. Moreover, the set of interior vertices in all $(v,u_i)$-paths is disjoint from all leaves in $\bigcup_{i=1}^{h_1} S(u_i)$. By the \emph{exterior} of the unit, denoted by $\mathsf{Ext}(F)$, we mean the set of all leaves in $\bigcup_{i=1}^{h_1}S(u_i)$. We call each $(v,u_i)$-path a \emph{branch} of $F$ and each edge in the star $S(u_i)$ a \emph{pendant} edge.
\end{defn}
\begin{figure}[h]
	\centering
	\includegraphics[width=4.5cm]{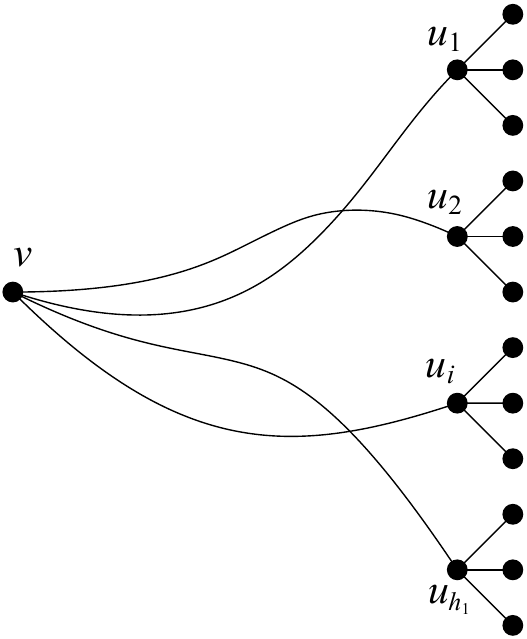}\\
	\caption{$(h_1,h_2,h_3)$-unit: $h_1$ vertex-disjoint stars $S(u_i)$ each of size $h_2$; $h_1$ edge-disjoint $(v,u_i)$-paths each of length at most $h_3$. Here $u_i$ may appear in a $(v,u_j)$-path for some $j\neq i$.}\label{}
\end{figure}

The following lemma guarantees a large collection of edge-disjoint units with distinct centers.

\begin{lemma}\label{lem-units}
	For each $0<\ep_1,\ep_2<1, \eta\ge\max\{\frac{40\ep_1}{\log 3},5\ep_2\}$ and $2\le s\le t$, there exists $C>0$ such that the following holds for all $n$ and $d$ with $d\ge \log^{200s}n$ and $n/d\geq C$. If $G$ is an $n$-vertex $K_{s,t}$-free $(\ep_1,\ep_2d)$-robust-expander with $d(G)=d$, then $G$ contains~$\ell'$ pairwise edge-disjoint $(\ell', m^5, 2m)$-units $F_1,\ldots, F_{\ell'}$ with distinct centers $v_1,\ldots,v_{\ell'}$, where $\ell'=(1-4\eta)d$ and  $m=\frac{2}{\ep_1}\log^3\left(\frac{15n}{\ep_2d}\right)$.
\end{lemma}

We first see how to construct a $K_{\ell}$-immersion, $\ell=(1-5\eta)d$, using Lemma~\ref{lem-units}.

\begin{proof}[Proof of Lemma~\ref{lem-immersion-not-dense} when $d\ge \log^{200s}n$.] Given constants $\ep_1,\ep_2$ and integers $t\ge s\ge 2$, we choose $d_0$ to be sufficiently large with foresight. Then by \eqref{lower} and \eqref{mnd}, we obtain that $m$ is also sufficiently large.
Let $F_1,\ldots, F_{\ell'}$ be the $(\ell',m^5,2m)$-units guaranteed by Lemma~\ref{lem-units}, with distinct centers $v_1,\ldots,v_{\ell'}$, where $\ell'=(1-4\eta)d$. Now we pick instead, for $\ell''=(1-4.5\eta)d$, a subfamily $\{F_1,\ldots, F_{\ell''}\}$ of the $(\ell',m^5,2m)$-units and then connect pairs of these $\ell''$ centers in the following way.

\stepcounter{propcounter}
\begin{enumerate}[label = {\bfseries \Alph{propcounter}\arabic{enumi}}]
       \item\label{youtiao1} Greedily connect as many pairs $(v_i,v_j)$ of centers as possible through an $(\mathsf{Ext}(F_i),\mathsf{Ext}(F_j))$-path $P_{i,j}$ of length at most $m$.

       \item\label{youtiao3} In each $F_i$, a star is \emph{occupied} if a leaf of it was previously used as an endpoint in an $(\mathsf{Ext}(F_i),\mathsf{Ext}(F_k))$-path for some $k\neq i$.

       \item\label{youtiao2} Let $\mathsf{Ext}(F_i)$ and $\mathsf{Ext}(F_j)$ be the current pair to connect. Then an $(\mathsf{Ext}(F_i),\mathsf{Ext}(F_j))$-path shall avoid using
          \begin{itemize}
            \item any leaf of the occupied stars in $F_i$ or $F_j$ as an endpoint;
          	\item edges that are used in previous connections;
          	\item  all centers $v_p$, $p\in[\ell'']$ and ;
          	\item edges that are in branches of units.
          \end{itemize}
\end{enumerate}
In the process, a star in a unit is \emph{over-used} if at least half of its pendant edges were used in previous connections. Discard a unit if it has at least $\eta d/4$ over-used stars.
Note that an $(\mathsf{Ext}(F_i),\mathsf{Ext}(F_j))$-path together with the corresponding branches within $F_i$ and $F_j$ form a $(v_i,v_j)$-path of length at most $6m$. Thus the total number of edges used in all connections is at most
\begin{eqnarray}\label{totale}
{\ell''\choose 2}\cdot 6m\le 3d^2m,
\end{eqnarray}
while the total number of edges in branches of all units is at most
\begin{eqnarray}\label{branches}
\ell''\ell'2m\le2d^2m.
\end{eqnarray}
Thus, in each connection, we avoid using a set of at most $5d^2m$ edges and a set of at most $\ell''$ centers $v_i, i\in [\ell'']$ in~\ref{youtiao2}.

We claim that there are at least $\ell=(1-5\eta)d$ units survived (were not discarded), say $F_1,\ldots, F_{\ell}$. Indeed, all units are edge-disjoint and each unit discarded has at least $\eta d/4\cdot m^5/2$ of its edges used in all connections. Hence by \eqref{totale}, the total number of units discarded is at most
$$\frac{3d^2m}{\eta d\cdot m^5/8}< \eta d/2\le \ell''-\ell,$$ where the first inequality follows as $m$ is sufficiently large.

We now claim that we can connect all pairs of the units $F_1,\ldots, F_{\ell}$ in~\ref{youtiao1}. Indeed, let $\mathcal{P}$ be a maximal collection of paths $P_{i,j}$ in~\ref{youtiao1} between different units $F_i$ with $i\in [\ell]$. Assume for contradiction that there exists $\{i,j\}\in {[\ell]\choose 2}$ such that there is no $(\mathsf{Ext}(F_i),\mathsf{Ext}(F_j))$-path  in $\mathcal{P}$. Note that each unit is connected with less than $\ell''$ other units in~\ref{youtiao1}. Thus each $F_i$ (or $F_j$) has at least $\ell'-\ell''=\eta d/2$ stars not occupied in~\ref{youtiao3}, among which there are at least $\eta d/4$ stars that are not over-used. Thus, $v_i$ and $v_j$ can still reach a subset of at least
$$\eta d/4\cdot m^5/2\ge dm^4=:x$$
vertices in $\mathsf{Ext}(F_i)$ and $\mathsf{Ext}(F_j)$, denoted by $L_i,L_j$ respectively. Recall that in~\ref{youtiao2}, by \eqref{totale} and \eqref{branches}, we need to avoid a set $P$ of at most $5d^2m\le d(G)\cdot \rho(x)\cdot x $ edges and at most $d\le \rho(x)\cdot x/4$ centers in each connection. Hence, Lemma~\ref{lem-diameter} applied to $G$ with \[X_1=L_i,~ X_2=L_j,~ Y=\{v_1,v_2,\ldots,v_{\ell''}\},~ F=P\] gives a desired $(L_i,L_j)$-path of length at most $m$, a contradiction to the maximality in~\ref{youtiao1}. Finally, notice that extending these paths in $\mathcal{P}$ from the exterior to the corresponding centers of the units yields edge-disjoint $(v_i,v_j)$-paths, for all $\{i,j\}\in {[\ell]\choose 2}$, which yields a desired $K_{\ell}$-immersion.
\end{proof}


\subsection{Proof of Lemma~\ref{lem-units}: finding units}\label{sec-unit}
Using the $K_{s,t}$-free condition, we first show that after deleting a small set of arbitrary vertices, the remaining subgraph still has large average degree.
\begin{prop}\label{densi}
For integers $n,s,t$ and constants $\eta$ with $0<\eta<1, t\ge s\ge 2$, there exists $m_0\in \mathbb{N}$ such that the following holds for every $m\ge m_0$. If $G$ is an $n$-vertex $K_{s,t}$-free graph with $d(G)=d\ge m^{50s}$, then for any subset $Z\subseteq V(G)$ with $|Z|\le dm^{50}$, we have $d(G-Z)\ge d-\eta d$.
\end{prop}
\begin{proof}
  We may assume that $|Z|\ge \eta d$ as otherwise it is trivial that $d(G-Z)\ge d-\eta d$. Consider the bipartite subgraph $G_1:=G[Z, V(G)\setminus Z]$. Since $G_1$ is $K_{s,t}$-free, by Lemma \ref{bipartite}, $e(G_1)\le c|Z|^{1-1/s}(n-|Z|)\le c (dm^{50})^{1-1/s}n\le \frac{cdn}{m^{50/s}}$, for some constant $c>0$, where the last inequality follows as $d\ge m^{50s}$. Also, $e(G[Z])\le\gamma|Z|^{2-1/s}\le \gamma(dm^{50})^{2-1/s}\le \frac{\gamma dn}{(dm^{50})^{1/s}}$,  for some $\gamma>0$, where the last inequality follows from \eqref{lower} and the assumption that $d\ge m^{50s}$. Thus
   $$d(G-Z)=\frac{nd-2e(G_1)-2e(G[Z])}{n-|Z|}\ge d-\frac{2cd}{m^{50/s}}-\frac{\gamma d}{d^{1/s}m^{50/s}}\ge d-\eta d,$$
  where the last inequality follows as long as $m$ is sufficiently large (with respect to $\eta$).
\end{proof}
Now we are ready to prove Lemma~\ref{lem-units}.
\begin{proof}[Proof of Lemma~\ref{lem-units}]

Let $\mathcal{F}=\{F_1, F_2,\ldots, F_k\}$ be a maximal collection of pairwise edge-disjoint $(\ell', m^5, 2m)$-units as desired . Let $Z$ be the centers of all units in $\mathcal{F}$ and $B$ be their edge set. Suppose for contrary that $|k|< \ell'$. Let $G':=(G-Z)\setminus B$, i.e.~the subgraph on vertex set $V(G)\setminus Z$ with edges in $B$ removed. Next we shall find in $G'$ one more $(\ell',m^5,2m)$-unit to reach a contradiction.

By \eqref{eq-n-large}, we have $|B|\le \ell'\cdot 2dm^5\le d^2m^6\le \frac{1}{2}\eta d(n-|Z|)$.
Together with Proposition \ref{densi}, we have
$$d(G')\ge d(G-Z)-2|B|/(n-|Z|)\ge d-2\eta d.$$
Then we claim that we can find in $G'$ vertex-disjoint stars, say $S(v_i)$ centered at some $v_i$, $i=1,\ldots,m^{10}$, each of size $d-3\eta d$ and $S(u_j)$ centered at $u_j$, $j=1,\ldots, dm^{15}$, each of size $m^{10}$. Indeed, let $U$ be the vertex set of a maximal collection $\mathcal{S}$ of vertex-disjoint of stars constructed as above. If $\mathcal{S}$ is not as desired, then $|U|\le dm^{30}\le d(G')m^{50}$ and Proposition \ref{densi} applied to $G'$ with $Z=U$ guarantees that $d(G'-U)\ge (1-\eta)d(G')\ge d-3\eta d$. This allows us to find one more star as desired, contradicting with the maximality of $\mathcal{S}$.

We use $L(v_i)$ to denote the set of all leaves in each star $S(v_i)$, $i\in\{1,\ldots,m^{10}\}$. Let $V=\{v_1,\ldots,v_{m^{10}}\}$. Now we shall use these vertex-disjoint stars to construct a new $(\ell',m^5,2m)$-unit in $G'$ as follows.
\stepcounter{propcounter}
\begin{enumerate}[label = {\bfseries \Alph{propcounter}\arabic{enumi}}]
    \item\label{B1} Connect as many pairs $(v_i,u_j)$ as possible through an $(L(v_i), u_j)$-path of length at most $m$, such that there is at most one path between any pair.

       \item\label{B2} For each $v_i$, a leaf $v\in L(v_i)$ is \emph{occupied} if it is previously used as an endpoint of a path in \ref{B1}.	
	\item\label{B3} Let $(v_i,u_j)$ be the current pair to connect. Then an $(L(v_i), u_j)$-path shall avoid using
          \begin{itemize}
            \item any leaf occupied in $L(v_i)$ as an endpoint;
	
	        \item edges in $\bigcup_{p=1}^{m^{10}}S(v_p)$, $\bigcup_{q=1}^{dm^{15}}S(u_q)$ and $B$;

             \item all vertices in $Z\cup V$.
       \end{itemize}

\end{enumerate}
\begin{claim}\label{new-unit}
  There is a vertex $v_i$ connected to at least $s=(\ell'+\eta d/2)$ distinct centers $u_j$.
\end{claim}

\begin{proof}
Suppose to the contrary that each $v_i$ is connected to less than $s$ centers $u_j$. Then the number of vertices used in all paths is at most $d\cdot m^{10} \cdot m= dm^{11}.$
Thus, there are at least $dm^{15}/2$ $u_j$-stars that are completely vertex-disjoint from all those paths, and there are at least $dm^{15}/2>dm^9=:x$ available centers from $u_j$-stars, say $U'$. Inside each $v_i$-star, there are at least $d(G)-3\eta d-s=\eta d/2$ leaves not occupied. Thus, there is a set $V'$ of at least $\eta d/2\cdot m^{10}>x$ leaves not occupied from all $v_i$-stars.

Recall that there are at most $d^2m^6$ edges in $B$, at most $dm^{11}$ edges used in all paths, at most $dm^{10}$ edges in $v_i$-stars and at most $dm^{25}$ edges in $u_j$-stars. Thus, in total, we avoid at most $d^2m^7\le d(G)\cdot \rho(x)\cdot x $ edges and at most $|Z|+|V|\le d+ m^{10}\le\rho(x)\cdot x/4$ vertices. Therefore, by Lemma~\ref{lem-diameter} with $X_1=U',X_2=V',Y=Z\cup V,F=B$, we can find a path of length at most $m$ between $U'$ and $V'$ in $G'$, resulting in one more pair of $v_i,u_j$ connected, a contradiction.
\end{proof}

Let $v_i$, $u_1,u_2,\cdots,u_s$ be the centers guaranteed by Claim \ref{new-unit} with all $(v_i,u_j)$-paths pairwise edge disjoint. If the set of interior vertices in all $(v_i,u_j)$-paths is disjoint from $\bigcup_{j=1}^{s} S(u_j)$, then they form a desired unit in $G'$. Otherwise, we discard a star $S(u_{j})$ if at least half of its leaves are used in  $(v_i,{u_{j}})$-paths. We claim that there are at least $\ell'$ $u_j$-stars left, say $S(u_1),\ldots, S(u_{\ell'})$. Indeed, recall that all $u_j$-stars are vertex-disjoint and the number of vertices in all all $(v_i,u_j)$-paths is at most $sm$. Thus the number of stars discarded is at most
$$\frac{sm}{m^{10}/2}\le \eta d/2\le s-\ell'.$$
Therefore, each $u_j$-stars left has at least $m^{10}/2\ge m^5$ leaves that are not used in a $(v_i,u_{j'})$-path for any $j' \in [s]$. These stars, together with the corresponding paths to $v_i$, form a desired unit in $G'$.

\end{proof}
\section{Embedding immersions in sparse expanders}\label{sec-sparse}

For the proof of the sparse case, we first deal with a special case when the maximum degree is somewhat bounded. Formally it is stated as follows.
\begin{lemma}\label{bounded max-degree}
Let $0<\ep_1\le1/400, 0<\ep_2<1/2, \eta\ge\max\{\frac{40\ep_1}{\log 3},5\ep_2\}$ and $2\le s\le t$. Then there exists $d_0$ such that the following holds for any $d\ge d_0$. Let $G$ be an $n$-vertex $K_{s,t}$-free $(\ep_1,\ep_2d)$-robust-expander with $d(G)=d< \log^{200s}n$ and $\Delta(G)\le d\log^{120}n.$ Then $G$ contains a clique immersion of order at least
     $\ell'=(1-4\eta)d$.
\end{lemma}
For the proof of this case, we need two technical lemmas which ensure that a small ball can expand robustly even after deleting a small set of vertices or edges.

\subsection{Technical lemmas: robust expansion of a ball}\label{subs40}

We first need the following lemma in \cite{hkl}, 
which ensures that a large set of vertices expands well even after deleting a small set of vertices.
\begin{lemma}[Proposition 3.5 in \cite{hkl}]\label{outer}
  Let $0<\ep_1\le1/400, 0<\ep_2<1/2$, and $G$ be an $n$-vertex $(\ep_1,\ep_2d)$-robust-expander with $d(G)= d$. If $X, Y\subseteq V(G)$ are  sets such that $|X|=x\ge \ep_2d, |Y|\le \rho(x)x/4$, then for every $i\le \log n$, $$|B^i_{G-Y}(X)|\ge\exp(\sqrt[4]{i}).$$
\end{lemma}
To state the second technical lemma, we need the following notion of consecutive shortest paths.
\begin{defn}
	For a set $X\subseteq W$ of vertices, the paths $P_1,\cdots, P_q$ are \emph{edge-disjoint consecutive shortest
		paths from $X$ within $W$} in $G$ if the following holds. For each $i\in [q]$, $P_i$ is a shortest path from $X$
	to a vertex in $W$ in the graph $G[W]\setminus\bigcup_{j\in[i-1]}E(P_j)$.
\end{defn}

\begin{lemma}\label{inner}
Let $0<\ep_1\le1/400, 0<\ep_2<1/2, \eta\ge\max\{\frac{40\ep_1}{\log 3},5\ep_2\}$. Then there exists $d_0:=d_0(\ep_1,\ep_2)$ such that the following holds for all $d\ge d_0$.
Suppose $G$ is an $n$-vertex $(\ep_1,\ep_2d)$-expander with $d(G)=d$ and $K,Z$ are disjoint sets of vertices with
$$|K|<n/2, ~ |Z|\le |K|\rho(|K|)/4, ~ |N_{G-Z}(K)|\ge \ell'+\ep_2d,$$
where $\ell'=(1-4\eta)d$.
Let $P_1,P_2,\cdots, P_q$ be edge-disjoint consecutive shortest
paths from $K$ within $G-Z$ such that $q<\ell'$ and denote $E=\bigcup_{j\in [q]}E(P_j)$. Then for any positive integers $t\in \mathbb{N}$ and $D$ with $D\le n/2$, we have \[|B^{t}_{(G-Z)\setminus E}(K)|\ge \min \left\{D, ~|K|\cdot \left(1+\frac{\ep_1}{2\log^2(15D/\ep_2d)}\right)^{t-1} \right\}.\]
\end{lemma}
\medskip

\begin{proof}
 Let $G'=G-Z$. For each $0\le p\le t-1$, let $X_p:=B^p_{G'\setminus E}(K)$ and denote by $E_p$ the set of edges in $E$ that go from the set $X_p$ to $N_{G'}(X_p)$. Then we have
 \begin{eqnarray}\label{aug}
 N_{G'\setminus E}(X_p)=N_{G'\setminus E_p}(X_p).
 \end{eqnarray}
 Note that only the first $p+2$ vertices of each $P_i$ can intersect $N_{G'}(X_p)$ for otherwise a shorter path can be found, contradicting the choices $P_1,\ldots,P_q$.
 Therefore $|E_p\cap E(P_i)|\le p+1$ for each $i\in [q]$ and it follows that for each $0\le p\le t-1$,
  \begin{eqnarray}\label{aug0}
  |E_p|\le q(p+1)<\ell'(p+1).
\end{eqnarray}
We first observe that for all $p\in[t-1]$, $|X_p|\ge |X_1|\ge |N_{G'}(K)|-|E_0|\ge\ep_2 d$, where we use \eqref{aug0} and the hypothesis of the lemma for the last inequality. Next we claim that for all $p\in[t-1]$, it holds that
 \begin{eqnarray}\label{aug1}
     |X_p|\cdot \rho(|X_p|)>p+1,
\end{eqnarray}
and we shall prove this later.
Indeed, if this holds, then we have that $|E_p|<\ell'(p+1)<d(G)\cdot|X_p|\cdot \rho(|X_p|)$ and it follows from the expansion property as in \eqref{expansion} that
 \begin{eqnarray}\label{aug2}
|N_{G\setminus E_p}(X_p)|\ge |X_p|\cdot \rho(|X_p|).
\end{eqnarray}
Since $|X_p|\ge |K|$ and $x\rho(x)$ is increasing in $x$, we obtain from the assumption on $|Z|$ that
 $$|Z|\le |K|\rho(|K|)/4<|X_p|\cdot \rho(|X_p|)/2,$$
 and it follows from \eqref{aug} and \eqref{aug2} that
 \begin{eqnarray}\label{half}
     |N_{G'\setminus E}(X_p)|=|N_{G'\setminus E_p}(X_p)|\ge |N_{G\setminus E_p}(X_p)|-|Z|\ge|X_p|\cdot \rho(|X_p|)/2.
\end{eqnarray}
We may assume $|X_p|\le D$ for all $1\le p\le t-1$, otherwise we are done. Now we have
  \begin{align}\label{eq1}
    |N_{G'\setminus E}(X_p)| &\nonumber
     \ge |X_p|\cdot \rho(|X_p|)/2=|X_p|\cdot \frac{\ep_1}{2\log^2(\frac{15|X_p|}{\ep_2d})} \\ \nonumber
     & \ge|X_p|\cdot\frac{\ep_1}{2\log^2\left(15D/\ep_2d\right)}, \nonumber
  \end{align}
  where the equality follows as $\ep_2 d\le |X_p|\le D$.
  Thus by definition, it follows that $|X_{p+1}|=|B^{1}_{G'\setminus E}(X_p)|=|X_p|+|N_{G'\setminus E}(X_p)|\ge|X_p|\left(1+\frac{\ep_1}{2\log^2(15D/\ep_2d)}\right)$ for each $p\in [t-1]$, and thus $$|X_t|\ge |K|\cdot \left(1+\frac{\ep_1}{2\log^2(15D/\ep_2d)}\right)^{t-1}.$$

Now it remains to prove \eqref{aug1}, which we will show by induction on $p$. Let $p_0$ be the least integer such that for each $p\ge p_0$, we have $p^2/4\cdot \rho(p^2/4)\ge p+1$. Then $p_0=O(\sqrt{d})$. The base cases $1\le p\le p_0$ easily follow since $|X_p|\cdot \rho(|X_p|)\ge|X_1|\cdot \rho(|X_p|)>\ep_2 d\cdot \rho(\ep_2 d)>p_0+1\ge p+1$ holds whenever $d$ is sufficiently large. Suppose $p>p_0$, and assume that \eqref{aug1} holds for all $p'$ with $1\le p'\le p-1$. Then $|E_{p'}|<\ell'(p'+1)<d(G)\cdot|X_{p'}|\cdot \rho(|X_{p'}|)$, which together with the expansion property from \eqref{half} implies that
\begin{align}
   |X_{p'+1}| &\nonumber
     \ge |X_{p'}|+|X_{p'}|\cdot \rho(|X_{p'}|)/2 \\ \nonumber
     & \ge|X_{p'}|+(p'+1)/2. \nonumber
\end{align}
Therefore, $|X_p|\ge|X_{1}|+\frac{2+3+\cdots+p}{2}\ge p^2/4$ and $$|X_p|\cdot \rho(|X_p|)\ge p^2/4\cdot \rho(p^2/4 )\ge p+1,$$
where the last inequality follows since $p>p_0$. This completes the proof of \eqref{aug1}.
\end{proof}

\subsection{Almost regular sparse expander: proof of Lemma~\ref{bounded max-degree}}\label{subs41}

Now we are ready to prove Lemma~\ref{bounded max-degree}. The proof idea is to choose vertices, say $v_1,v_2,\cdots,v_{\ell'}$, that are pairwise far apart to be the branch vertices of our clique immersion. To achieve this, we grow two nested balls around each $v_i$, one inner ball $B^r(v_i)$ and one outer ball $B^{\kappa+r}(v_i)$ for integers $r\ll \kappa$. Then we try to connect all pairs $v_i, v_j$ using a shortest path between the outer balls around them while avoiding (1) all edges in the inner balls of other branch vertices, that is, $\bigcup_{p\neq i,j}E(B^r(v_p))$; (2) all edges used in previous connections. Using the robust expansion guaranteed by the Lemma \ref{outer} and Lemma \ref{inner}, we are able to regrow new inner and outer balls around each $v_i$ to be large enough to enable us to connect $v_i$ to more branch vertices.\medskip
\begin{proof}[Proof of Lemma~\ref{bounded max-degree}]
Let $\kappa=\lceil\log n/(800s\log\log n)\rceil$, $r=\lceil(\log\log n)^5\rceil$ and $d_0$ be a sufficiently large integer. First, we claim that there are at least $d$ vertices of degree at least $d(G)-2\eta d$ which are pairwise a distance at least $3\kappa+1$ far apart. Indeed, letting  $L$ be the set of vertices with degree at least $d(G)-2\eta d$, as $\Delta(G)\le d\log^{120}n$, we see that
\[d(G)n<(d(G)-2\eta d)(n-|L|)+d\log^{120}n\cdot|L|,\]
implying that $|L|> n/\log^{121}n$. Let $Y$ be a maximal set of vertices in $L$ which are pairwise a distance at least $3\kappa+1$ apart. Suppose to the contrary $|Y|<d$. Since $d<\log^{200s}n$ and $\Delta(G)\le d\log^{120}n$, we have that $$|B^{3\kappa}(Y)|< 2d\Delta(G)^{3\kappa}<\log^{700s\kappa}n<n^{7/8}<n/\log^{121}n,$$ where the last inequality follows as $d_0$, also $n$, is sufficiently large. Thus there exists $v\in L\setminus B^{3\kappa}(Y)$, contradicting the maximality of $Y$.

Let then $v_1,v_2,\cdots,v_{\ell'}$ be such vertices in $Y$, which will serve as the branch vertices of the clique immersion. By choice, all the balls $B^{\kappa}_G(v_i)$, $i\in[\ell']$, are pairwise vertex disjoint. Let $I\subseteq\binom{[\ell']}{2}$ be a maximal subset for which we can find paths $P_{e}, e\in I$, so that the following hold.
\stepcounter{propcounter}
\begin{enumerate}[label = {\bfseries \Alph{propcounter}\arabic{enumi}}]
       \item\label{qingzhenguihuayu1} For each $\{i,j\}\in I$, $P_{ij}$ is a $(v_i, v_j)$-path with length at most $2\log^4n$.
       \item\label{qingzhenguihuayu2} For distinct $e, e'\in I$, the paths $P_e$ and $P_{e'}$ are edge disjoint.

       \item\label{qingzhenguihuayu3} For each $e\in I$ and $i\notin e$, $E(B^r_G(v_i))$ and $E(P_e)$ are disjoint.

       \item\label{qingzhenguihuayu4} For each $i\in[\ell']$, the subcollection $P_e$, $e\in I$ with $i\in e$, form edge-disjoint consecutive shortest paths from $v_i$ within $B_G^r(v_i)$.
\end{enumerate}
If $I=\binom{[\ell']}{2}$, then by \ref{qingzhenguihuayu2}, we have a $K_{\ell'}$-immersion with branch vertices $v_1,v_2,\cdots,v_{\ell'}$. Suppose there exists some $\{i,j\}\in \binom{[\ell']}{2}\setminus I$. Let $W=\bigcup_{e\in I}E(P_e)$. Then by \ref{qingzhenguihuayu1} and \ref{qingzhenguihuayu2}, we have
 \[|W|< \binom{\ell'}{2}\cdot 2\log^4n\le d^2\log^4n.\]
By~\ref{qingzhenguihuayu4}, we can apply Lemma \ref{inner} on $G$ with $t=r, ~D=d^2\log^7 n, ~K=\{v_i\}, ~Z=\varnothing$ and $E=\bigcup\limits_{i\in e\in I}E(P_e)$ to get that each inner ball robustly maintains size
\[|B^r_{G\setminus W}(v_i)|\ge \min \left\{D, ~|K|\cdot \left(1+\frac{\ep_1}{2\log^2(15D/\ep_2d)}\right)^{r-1} \right\}\ge d^2\log^7 n.\]
Furthermore,
by Lemma \ref{outer} with $X=B^r_{G\setminus W}(v_i)$ and $Y=\bigcup_{e\in I}V(P_e)$,  we have \[|B^{\kappa+r}_{G\setminus W}(v_i)|\ge |B^{\kappa}_{G-Y}(X)|\ge\exp(\sqrt[4]{\kappa})\ge\exp(\sqrt[5]{\log n})=:x.\] Similarly $|B^{\kappa+r}_{G\setminus W}(v_j)|\ge x$.
Let $W_{i,j}=W\cup(\bigcup_{p\neq i,j}E(B^r_{G}(v_{p})))$, that is, the set of all edges that are either used in
some connection or are in some inner ball of other branch vertices $v_p$. As we chose the vertices $v_i$ to
be pairwise at least a distance $3\kappa+1> 2\kappa + 2r$ apart, all the outer balls $B^{\kappa+r}_{
G\setminus W}(v_i)$ are pairwise vertex-disjoint. Note that \[|W_{i,j}|\le|W| + \De(G)\cdot 2(1-4\eta)d\cdot \De(G)^r<  d(G)\cdot\rho(x)x .\]
Therefore, Lemma \ref{lem-diameter} applied to $G$ with $X_1=B^{\kappa+r}_{
G\setminus W}(v_i),X_2=B^{\kappa+r}_{
G\setminus W}(v_j)$ and $F=W_{i,j}$, gives a ($B^{\kappa+r}_{
G\setminus W}(v_i),B^{\kappa+r}
_{G\setminus W}(v_j)$)-path in $G\setminus W_{i,j}$ with length at
most $m\le \log^4n$, which can be extended, within the balls $B^{\kappa+r}
_{G\setminus W}(v_i)$ and $B^{\kappa+r}
_{G\setminus W}(v_j)$, into a
$(v_i,v_j)$-path. Thus, if we let $P_{ij}$ be a shortest $v_i, v_j$-path in $G\setminus W_{i,j}$, then $P_{ij}$ has length at
most $\log^4n+2(\kappa+r)\le2\log^4n$. The path $P_{ij}$, together with all $P_e, e\in I$, satisfies the conditions \ref{qingzhenguihuayu1}--\ref{qingzhenguihuayu4} above, contradicting the maximality of $I$.
\end{proof}
\subsection{Finishing the proof of Lemma~\ref{lem-immersion-not-dense}}\label{subs43}

This subsection is devoted to finishing the proof of Lemma~\ref{lem-immersion-not-dense} when $d< \log^{200s} n$. As sketched in Section~\ref{subs22}, we divide the proof into two cases depending on the number of vertices of relatively large degree. In particular, let $Z_1=\{v\in V(G)\mid d(v)\ge dm^3\}$ be the set of high degree vertices. We first deal with the case that $|Z_1|\ge d$ (see Claim \ref{many-large}). If $|Z_1|<d$, then we show in Claim \ref{cl1} that the subgraph $G':=G-Z_1$ still has large average degree and we focus our attention to $G'$ which additionally has small maximum degree. We then find in $G'$ a collection of small $(\ep_1,\ep_2d)$-robust-expanders $F_1,F_2,\cdots,F_d$ which are pairwise far apart in $G'$ (see Claim~\ref{manyexp}). In Section~\ref{kernels}, we first grow, inside each subexpander $F_i$, a vertex $v_i$ robustly into a small ball $K_i$, called \emph{kernel} (See Claim~\ref{roubao}). In this case, each $K_i$, locally maintains certain robust expansion property inside $F_i$. Moreover, each kernel has large enough size so as to enjoy further expansion in the original expander $G$. Then Section~\ref{subs44} is devoted to building a desired clique immersion by further expanding (subsets of) each kernel into two large balls in $G'$,  and here we follow the strategy in the proof of Lemma~\ref{bounded max-degree}.

\subsubsection{Bounding the size of $Z_1$}\label{subs42}

The following claim builds a large clique immersion with many vertices of relatively large degree. Recall that $Z_1=\{v\in V(G)\mid d(v)\ge dm^3\}$.
\begin{claim}\label{many-large}
  If $|Z_1|\ge d$, then $G$ contains a $K_d$-immersion.
\end{claim}
\begin{proof}
 Let $v_1,v_2,\cdots,v_d$ be distinct vertices in $Z_1$. We shall construct a $K_{d}$-immersion with all $v_i$ as branch vertices in the following way, where we let $N_i=N_G(v_i)$ for $i\in[d]$ and $I=\{\{i,j\}\in \binom{[d]}{2}\mid v_iv_j\notin E(G)\}$.
\stepcounter{propcounter}
\begin{enumerate}[label = {\bfseries \Alph{propcounter}\arabic{enumi}}]
       \item\label{mianbao1} Greedily connect as many pairs $v_i,v_j$ with $\{i,j\}\in I$ as possible through an $(N_i,N_j)$-path of length at most $m$;
       \item\label{mianbao2} Let $(N_i,N_j)$ be the current pair to connect. Then avoid using edges that are used in previous connections and all other vertices $v_p$, $p\neq i,j$.
\end{enumerate}
Note that an $(N_i,N_j)$-path together with the corresponding incident edges forms a $(v_i,v_j)$-path, say $P_{i,j}$, of length at most $m+2<2m$.
Thus the total number of edges used in all connections is at most ${d\choose 2}\cdot 2m\le d^2m$. It remains to show that there exist, for all pairs $\{i,j\}\in I$, pairwise edge-disjoint paths $P_{i,j}$ as above. Indeed, throughout the process, each $v_i$ always has at least $dm^3-d\ge dm^3/2$ incident edges which are not used in all previous paths. Thus, each $v_i$ can still reach a set of at least
$dm^3/2$ neighbors in $N_i$, denoted by $N_i'$. Let $x:=dm^3/2$. Then there are at most $d^2m\le d(G)\cdot \rho(x)\cdot x $ edges and exactly $d\le \rho(x)\cdot x/4$ vertices we need to avoid in each connection. Thus any $\{i,j\}\in I$, by Lemma~\ref{lem-diameter} applied with $X_1=N_i', X_2=N_j',Y=Z_1$ and $F$ being the set of edges used in all previous paths, there exists an $(N_i,N_j)$-path of length at most $m$ as desired.
\end{proof}

Thus, it remains to consider the case $|Z_1|< d$. The following claim guarantees that the subgraph $G- Z_1$ still has large average degree, which allows us to restrict our attention to the subgraph $G- Z_1$.

\begin{claim}\label{cl1}
 $d(G-Z_1)\ge d(G)-\eta d$.
\end{claim}
\begin{proof}
  We may assume that $|Z_1|\ge \frac{\eta d}{2}$, otherwise we are done. It suffices to show that $e(G[Z_1])+e(Z_1,V(G)-Z_1)\le \eta dn/2$. Indeed, it is easy to see that $e(G[Z_1])\le |Z_1|^2=o(dn)$ because $d<\log^{200s}n$. Also, by $K_{s,t}$-freeness, Lemma \ref{bipartite} implies that $e(Z_1,V(G)\setminus Z_1)=O(d^{1-1/s}n)=o(dn)$.
\end{proof}

\subsubsection{Finding small dense expanders in $G-Z_1$}\label{manyexpander}
Let $G'=G-Z_1$. Then Claim \ref{cl1} tells that $d(G')\ge d(G)-\eta d$. We now proceed to find small expanders of large average degree that are pairwise far apart from each other in $G'$.
Recall that $\kappa=\lceil\log n/(800s\log\log n)\rceil$. Let $\mathcal{F}$ be a maximal family of subgraphs in $G'$ satisfying the following.

\stepcounter{propcounter}
\begin{enumerate}[label = {\bfseries \Alph{propcounter}\arabic{enumi}}]
  \item Each $F\in\mathcal{F}$ is an $(\ep_1,\ep_2d)$-robust-expander with $d(F)\ge (1-3\eta)d$.
  \item\label{kongxincai} For distinct $F,F'\in\mathcal{F}$, $B_{G'}^{\kappa}(V(F))\cap B_{G'}^{\kappa}(V(F'))=\varnothing$.
\end{enumerate}
For each $F\in\mathcal{F}$, let $n_F=|F|$, $m_F=\frac{2}{\ep_1}\log^3\left(\frac{15n_F}{\ep_2d}\right)$.  If $d(F)\ge \log^{200s}n_F$ or $\Delta(F)\le d\log^{120}n_F$, then by the proof of the dense case in Section $3$, or by Lemma \ref{bounded max-degree}, we have a clique immersion of order at least $(1-5\eta)d(F)\ge(1-9\eta)d$, finishing the proof. Thus we may assume that for each $F\in\mathcal{F}$, $$d(F)<\log^{200s}n_F, \quad \Delta(G')\ge\Delta(F)\ge d\log^{120}n_F.$$ Recall that $\Delta(G')\le dm^3\le d\log^{12}n$, therefore
\begin{equation}\label{xiangchiyu}
	\exp(\sqrt[200s]{d/2})~\le~ n_F~\le ~\exp(\sqrt[10]{\log n}).
\end{equation}

\begin{claim}\label{manyexp}
It holds that $|\mathcal{F}|\ge d$.
\end{claim}
\begin{proof}
  Otherwise, let $U=\bigcup_{F\in\mathcal{F}}B_{G'}^{2\kappa}(V(F))$. Then $$|U|\le |\cF|\cdot 2\max_{F\in\cF}n_F\cdot \De(G')^{2\kappa}\le d\exp(\sqrt[10]{\log n})(dm^3)^{2\kappa}<n/m^4$$ and it follows that $d(G'-U)\ge d(G')-2|U|\cdot dm^3/n\ge (1-2\eta)d$. Lemma \ref{lem-expander} implies that $G'-U$ contains an $(\ep_1,\ep_2d)$-robust-expander $F'$ with $d(F')\ge (1-3\eta)d$, which, by choice, is far from all expanders in $\cF$ at a distance of at least $2\kappa+1$, contradicting the maximality of $\mathcal{F}$.
\end{proof}

\subsubsection{Kernels in subexpanders}\label{kernels}

As shown in Claim~\ref{manyexp}, $\mathcal{F}$ contains at least $d$ expanders, say $F_1,F_2,\cdots, F_d$. Since each $F_i$ satisfies $d(F_i)\ge (1-3\eta)d$, we can choose distinct vertices $v_1,v_2,\cdots,v_d$ as the branch vertices such that $v_i\in V(F_i)$ and $d_{F_i}(v_i)\ge (1-3\eta)d, i\in[d]$. 
Note that for sufficiently large $d$, by~\eqref{xiangchiyu}, each small expander $F_i
$ has at least $\exp(\sqrt[200s]{d/2})\ge d^2$ vertices. For each $v_i$, we shall show that, it expands robustly in $F_i$ as follows.
\begin{claim}\label{roubao}
Let $P_1,P_2,\cdots, P_q$ be edge-disjoint consecutive shortest
paths from $v_i$ in $F_i$ for any $q<(1-4\eta)d$ and denote $E=\bigcup_{j\in [q]}E(P_j)$. Then it holds that $|B^{s}_{F_i\setminus E}(v_i)|\ge d^2$ for any integer $s\ge \log^4d$.
\end{claim}
\begin{proof}
  To see this, as $d_{F_i}(v_i)\ge (1-3\eta)d\ge \ell'+\ep_2d$, by applying Lemma \ref{inner} with $G=F_i$, $t=\log^4d$, $D=d^2, K=\{v_i\}$ and $Z=\varnothing$, we have for large $d$ that \[|B^t_{F_i\setminus E}(v_i)|\ge \min \left\{D, ~ \left(1+\frac{\ep_1}{2\log^2(15D/\ep_2d)}\right)^{t-1} \right\}\ge d^2.\]
\end{proof}

For each $i\in[d]$ and $s=\log^4d$, we call $K_i=B^{s}_{F_i}(v_i)$ the \emph{kernel} for $v_i$. In Claim~\ref{roubao}, we see that the vertex $v_i$ locally expand robustly in $F_i$. Recall that $r=\lceil(\log\log n)^5\rceil$. By~\ref{kongxincai}, we observe that all balls $B^{r}_{G'}(K_i)$ are pairwise disjoint.

\begin{claim}\label{roubao1}
Let $K'\subseteq K_i$ be a subset of size at least $d^2$ and $P_1,P_2,\cdots, P_q$ be edge-disjoint consecutive shortest
paths from $K'$ within $B^{r}_{G'}(K')$ in $G'$ such that $q<\ell'$ and denote $E=\bigcup_{j\in [q]}E(P_j)$. Then $|B^r_{G'\setminus E}(K')|\ge d^2\log^7 n.$
\end{claim}
 Indeed, such $K'$ satisfies $|K'|\cdot \rho(|K'|)/4\ge d>|Z_1|$ and then $N_{G'}(K')\ge d\ge \ell'+\ep_2 d$. Therefore, Claim~\ref{roubao1} follows from Lemma \ref{inner} applied to $G$ with $t=r, D=d^2\log^7 n, K=K'$ and $Z=Z_1$.

\subsubsection{Building a clique immersion}\label{subs44}
We are now ready to build a clique immersion by iteratively finding edge disjoint paths connecting all pairs of branch vertices $v_1,v_2,\cdots,v_{\ell'}$. To do this, we will further expand each kernel $K_i$ in two stages (to $B_{G'}^r(K_i)$ and then $B_{G'}^{\kappa+r}(K_i)$ as depicted in Figure~\ref{suanrongshanbei}).

\begin{figure}[h]
	\centering
	\includegraphics[width=11.5cm]{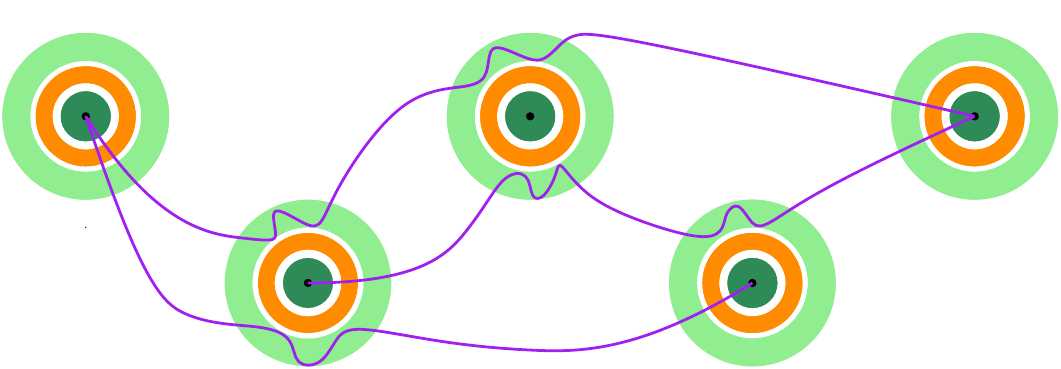}\\
	\caption{Each $v_i$ is surrounded by three layers: the kernel $K_i=B^{s}_{F_i}(v_i)$ for $s=\log^4 d$, the inner ball $B^r_{G'}(K_i)$ (the orange one) and the outer ball $B^{\kappa+r}_{G'}(K_i)$. }\label{suanrongshanbei}
\end{figure}

By property~\ref{kongxincai}, all the balls $B^{\kappa}_{G'}(K_i)$ are pairwise disjoint, $i\in[\ell']$. Let $I\subset\binom{[\ell']}{2}$ be a maximal subset for which we can find pairwise edge-disjoint paths $P_{e}, e\in I$, so that
\stepcounter{propcounter}
\begin{enumerate}[label = {\bfseries \Alph{propcounter}\arabic{enumi}}]
     \item\label{woxianghetang30} for each $e=\{i,j\}\in I$, $P_e$ is a $(v_i, v_j)$-path with length at most $2\log^4n$;
     \item\label{woxianghetang3} for each $e\in I$, $P_e$ is disjoint from the ball $B^r_{G'}(K_i)$ for any $i\notin e$;
       \item\label{woxianghetang4} for each $i\in[\ell']$, the subcollection $P_e$, $e\in I$ with $i\in e$, form edge-disjoint consecutive shortest paths from $v_i$ within $K_i$, and $P_e-K_i$, $e\in I$ with $i\in e$, are edge-disjoint consecutive shortest paths from $K_i$ within $B_{G'}^{r}(K_i)$.
\end{enumerate}

We may assume there is $\{i,j\}\in \binom{[\ell']}{2}\setminus I$ for otherwise all paths $P_e$, $e\in I$, would form a desired clique immersion with $v_1,v_2,\cdots,v_{\ell'}$ as branch vertices. We will reach a contradiction by finding a $(v_i,v_j)$-path that is short and additionally avoids all the edges used in previous connections and all vertices in the inner balls of other branch vertices $v_{p}, p\neq i,j$ as in \ref{woxianghetang3}.

Let $W=\bigcup_{e\in I}E(P_e)$ and $U=\bigcup_{e\in I}V(P_e)$, i.e. the sets of edges and vertices used in all previous connections, respectively. Then $|W|,|U|\le  d^2\log^4n$. By~\ref{woxianghetang30},~\ref{woxianghetang3} and~\ref{woxianghetang4}, Claim~\ref{roubao} implies that for $i\in [\ell']$, $|B^{s}_{F_i\setminus W}(v_i)|\ge d^2.$ Let $K_i'=B^{s}_{F_i\setminus W}(v_i)$, $i\in [\ell']$. Then it follows from \ref{woxianghetang3}, \ref{woxianghetang4} and Claim~\ref{roubao1} that $|B^r_{G'\setminus W}(K_i')|\ge d^2\log^7 n$.

Next, by Lemma \ref{outer} with $X=B^r_{G'\setminus W}(K_i'), Y=(U\setminus\{v_i\})\cup Z_1$, we have $$|B^{\kappa+r}_{G'\setminus W}(K_i')|\ge|B^{\kappa}_{G-Y}(X)|\ge\exp(\sqrt[4]{\kappa})\ge \exp(\sqrt[5]{\log n})=:x.$$
Let $U^*=\bigcup_{p\neq i,j}B^r_{G'}(K_p)$. As we choose the kernels $K_i$ to
be pairwise at least a distance $2\kappa+1> \kappa+ 2r$ apart, both $B^{\kappa+r}_{
G'\setminus W}(K_i')$ and $B^{\kappa+r}
_{G'\setminus W}(K_j')$ are disjoint
from $U^*$. Recall from~(\ref{xiangchiyu}) that $|K_i|\le |F_i|\le \exp(\sqrt[10]{\log n})$, for each $i\in[\ell']$ and $\Delta(G')\le dm^3$. Thus the total number of vertices we avoid in~\ref{woxianghetang3}  is \[|U^*|\le(1-4\eta)d\cdot2\exp(\sqrt[10]{\log n})(dm^3)^r\le \rho(x)x/4.\]
Moreover, the number of edges we avoid in \ref{woxianghetang3} is $|W|\le d^2\log^4n\le d(G)\rho(x)x $.
Since $|B^{\kappa+r}
_{G'\setminus W}(K_i')|, \\~|B^{\kappa+r}_{
G'\setminus W}(K_j')|\ge x,$ by Lemma~\ref{lem-diameter} with \[X_1=B^{\kappa+r}
_{G'\setminus W}(K_i'),~X_2=B^{\kappa+r}_{
G'\setminus W}(K_j'),~Y=U^*~\text{and}~F=W,\] there exists a ($B^{\kappa+r}_{
G'\setminus W}(K_i'),B^{\kappa+r}
_{G'\setminus W}(K_j')$)-path, say $Q_{ij}$, of length at most $m$ in $(G-U^*)\setminus W$. It is easy to observe that $Q_{ij}$ can be extended, first within the ball $B^{\kappa+r}
_{G'\setminus W}(K_i')$ (or $B^{\kappa+r}
_{G'\setminus W}(K_i')$) and then within $K_i'=B^{s}_{F_i\setminus W}(v_i)$ (or $K_j'$), into a
$(v_i,v_j)$-path, in which we denote by $P_{ij}$ such a shortest $(v_i, v_j)$-path in $(G-U^*)\setminus W$. Then $P_{ij}$ always has length at
most $m+2\kappa+2r+2s\le2\log^4n$, which together with the paths $P_e, e\in I$ satisfy~\ref{woxianghetang30},~\ref{woxianghetang3} and \ref{woxianghetang4}, contradicting the maximality of $I$.\medskip

This completes the proof.

\section*{Acknowledgements}
We thank the referee for the careful readings and suggestions that improve the presentation.
\bibliographystyle{abbrv}
\bibliography{ref}


\end{document}